\documentclass[
12pt,oneside,english,a4paper]{amsart}
\usepackage[T1]{fontenc}
\usepackage[utf8]{inputenc}
\usepackage[dvipsnames]{xcolor}
\usepackage[yyyymmdd,hhmmss]{datetime}
\usepackage[numbers,sort,compress]{natbib}
\usepackage{amstext,amsthm,amssymb}
\usepackage{stmaryrd}
\usepackage{xparse,mathrsfs,mathtools}
\usepackage{tikz}
\usetikzlibrary{backgrounds,patterns}
\usepackage[english]{babel}
\usepackage{xr-hyper}
\usepackage[pdftex,unicode=true,%
pdfusetitle,pdfdisplaydoctitle=true,%
pdfpagemode=UseOutlines,%
bookmarks=true,bookmarksnumbered=false,bookmarksopen=true,bookmarksopenlevel=2,%
breaklinks=true,%
pdfencoding=unicode,psdextra,
pdfcreator={},
pdfborder={0 0 0},
colorlinks=false
]{hyperref}
\usepackage{lmodern,microtype}
\usepackage{subcaption}
\usepackage[dvipsnames]{xcolor}
\usepackage{geometry}
\usepackage[capitalise,nameinlink]{cleveref}
\usepackage[cal=zapfc,bb=ams,frak=esstix,scr=boondoxo]{mathalfa}

\hypersetup{
	colorlinks=false,
	pdfborder={0 0 0},
	pdfborderstyle={/S/U/W 0},
}

\newtheorem*{Theorem*}{Theorem}
\newtheorem{Theorem}{Theorem}[section]
\theoremstyle{remark}
\newtheorem*{Remark*}{Remark}

\NewDocumentCommand\e{ s O{} m }{%
	\IfBooleanTF{#1}{%
		\operatorname{e}_{#2}\parentheses*{#3}%
	}{\operatorname{e}_{#2}\parentheses{#3}}%
}

\DeclarePairedDelimiter\parentheses{\lparen}{\rparen}
\DeclarePairedDelimiter\braces{\lbrace}{\rbrace}

\NewDocumentCommand\set{ s o m o }{%
	\IfBooleanTF{#1}{\IfNoValueTF{#4}{\braces*{#3}}{\braces*{\,#3:#4\,}}}{%
		\IfNoValueTF{#2}{\IfNoValueTF{#4}{\braces{#3}}{\braces{\,#3:#4\,}}}{%
			\IfNoValueTF{#4}{\braces[#2]{#3}}{\braces[#2]{\,#3:#4\,}}}}%
}

\crefname{section}{§}{§§}
\crefname{figure}{Figure}{Figures}
\crefformat{equation}{#2(#1)#3}
\crefformat{enumi}{#2(#1)#3}
\numberwithin{equation}{section}
\makeatletter
\renewcommand\p@subfigure{\thefigure~}
\makeatother

\usepackage{xifthen,ifthen}
\makeatletter
\newcounter{@ToDo}
\newcommand{\todo@helper}[1]{%
	({\color{blue}TODO~\arabic{@ToDo}: {#1\@addpunct{.}}})%
}
\newcommand{\todo}[1]{\stepcounter{@ToDo}%
	\relax\ifmmode\text{\todo@helper{#1}}%
	\else\todo@helper{#1}\fi%
}
\makeatother

\title[Continuous and discrete universality]{Continuous and discrete universality of zeta-functions: Two sides of the same coin?}
\date{\today{}}

\makeatletter\@namedef{subjclassname@2020}{\textup{2020} Mathematics Subject Classification}\makeatother
\subjclass[2020]{
	Primary
	11M06,
	11M35; 
	Secondary
	47Axx, 
    37B20 
}
\keywords{Riemann zeta-function, Hurwitz zeta-function, universality, linear dynamics, strong recurrence}

\author{Athanasios~Sourmelidis}

\address{
	Athanasios~Sourmelidis\\%
	Institut für Analysis und Zahlentheorie\\%
	TU~Graz\\%
	Steyrergasse 30\\%
	8010~Graz\\%
	Austria}
\email{sourmelidis@math.tugraz.at}

\thanks{%
	The author is supported by FWF project~M~3246-N
}

\begin{document}
\begin{abstract}
	In 1975 Voronin proved the universality theorem for the Riemann zeta-function $\zeta(s)$ which roughly says that any admissible function $f(s)$ is approximated by $\zeta(s)$. 
	A few years later Reich proved a discrete analogue of this result.
	The proofs of these theorems are almost identical but it is not known whether  one of them implies the other.
	We will see that if we translate the question in the language of linear dynamics  then there is a link which we exploit to obtain in a straightforward way a big variety of discrete universality results appearing in the literature.
\end{abstract}
\maketitle
\section{Introduction and Main results}
\subsection{Introduction}
The Riemann zeta-function  $\zeta(s)$ is defined for every complex number $s:=\sigma+it$ with $\sigma>1$ by the following two expressions
\[
\zeta(s):=\sum_{n\geq1}\frac{1}{n^s}=\prod_{p}\parentheses*{1-\frac{1}{p^s}}^{-1}.
\]
In his pathbreaking memoir \cite{Riemann}, Riemann showed that $\zeta(s)$ has a meromorphic continuation to the whole complex plane $\mathbb{C}$ with a simple pole at $s=1$ and that it satisfies a certain functional equation  connecting $\zeta(s)$ with $\zeta(1-s)$.
This in turn implies that the only real zeros of $\zeta(s)$, the so called {\it trivial zeros}, are the negative even integers and Riemann conjectured that all complex zeros of $\zeta(s)$, the so called {\it non-trivial zeros}, lie on the vertical line $1/2+i\mathbb{R}$ (Riemann hypothesis).

In 1975 Voronin \cite{Voronin1975} proved the following:

{\it If $0<r<1/4$ and $f:\lbrace s\in\mathbb{C} :|s|\leq r\rbrace\to\mathbb{C}$ is a non-vanishing continuous function which is analytic in the interior of the disk, then for every $\varepsilon>0$ there is $\tau>0$ such that}
\begin{align}\label{Voroninuniversailty}
\max_{|s|\leq r}|\zeta(s+3/4+i\tau)-f(s)|<\varepsilon.
\end{align}
In the proof it is already implicit that the set of those $\tau>0$ satisfying \eqref{Voroninuniversailty} has positive lower density in the non-negative real numbers $\mathbb{R}_+$. 
Recall that a set $A\subseteq\mathbb{R}_+$ has {\it positive lower density} if
\[
\liminf_{T\to\infty}\frac{1}{T}\mathrm{meas}(A\cap[0,T])>0,
\]
where $\mathrm{meas}(A)$ denotes the Lebesgue measure of the set $A$.

Since $\zeta(s)$ approximates a large family of analytic functions, the above result has been known ever since as the {\it (continuous) universality theorem} for $\zeta(s)$.
Voronin's theorem has been refined and generalized within the decades for other zeta- and $L$-functions. 
There is an exhaustive literature of results indicating that the universality property is rather the norm than the exception when one deals with a Dirichlet series satisfying some natural (for applications to number theory) conditions.
We refer to  \cite{Queffelecbook,Laurincikasbook,Matsumotosurvey,Steudingbook} for a survey of such results.

In 1979/1981 Gonek \cite{Gonekthesis} and (independently) Bagchi \cite{Bagchithesis, Bagchi1982} generalized Voronin's  theorem in several directions. 
For instance, they improved his theorem by replacing the disc centered at $3/4$ and having radius $0<r<1/4$, by any compact set $K$ of the strip $\mathcal{D}:=\lbrace s\in\mathbb{C}:1/2<\sigma<1\rbrace$ that has connected complement.
\begin{Remark*}
	We take the chance here to introduce some notation in order to simplify the upcoming exposition.
	A tuple $(K,f,\varepsilon)$ will be called {\it admissible} if $K\subseteq\mathcal{D}$ is compact with connected complement, $f:K\to\mathbb{C}$ is continuous in $K$ and analytic in its interior and $\varepsilon>0$.
	A tuple $(K,f,\varepsilon)^*$ will also be called admissible if the previous assumptions hold in addition to $f$ being non-vanishing in $K$. 
\end{Remark*}

Gonek considered another generalization of $\zeta(s)$, namely the Hurwitz zeta-function, which is defined by
\[
\zeta(s;\alpha):=\sum_{n\geq0}\frac{1}{(n+\alpha)^s},\quad\sigma>1,\quad\alpha\in(0,1].
\]
Like $\zeta(s)$, the Hurwitz zeta-function has a meromorphic continuation to  $\mathbb{C}$ with a simple pole at $s=1$.
Gonek showed that if the parameter $\alpha$ is a fixed transcendental or rational number $\neq1/2,1$ then the following holds: 

{\it If $(K,f,\varepsilon)$ is admissible, then}
\begin{align}\label{Hurwitzuniversal}
\liminf_{T\to\infty}\frac{1}{T}\mathrm{meas}\left\{\tau\in[0,T]:\max_{s\in K}|\zeta(s+i\tau;\alpha)-f(s)|<\varepsilon\right\}>0.
\end{align}
Observe that in this case the target function $f$ is allowed to have zeros in $K$ and this is why we talk about {\it (continuous) strong universality}.

We will see later on that this property of the Hurwitz zeta-function is equivalent to saying that the set $\lbrace\zeta(\cdot+i\tau;\alpha):\tau\geq0\rbrace$
is dense in the space of analytic functions $f:\mathcal{D}\to\mathbb{C}$, $H(\mathcal{D})$, equipped with the topology of uniform convergence on compact sets.
In particular, this set intersects any open subset of $H(\mathcal{D})$ {\it frequently often}  as the positive lower density statement implies (we will give a precise formulation of what this means in the following section).
In this context and in view of Voronin's theorem, we also have that $\lbrace\zeta(\cdot+i\tau):\tau\geq0\rbrace$ intersects frequently often any open subset of the subspace $H^*(\mathcal{D}):=\lbrace f\in H(\mathcal{D}):f\text{ is non-vanishing}\rbrace$ which is endowed with the induced topology.

At the same time with Gonek and Bagchi, Reich \cite{Reich1980} showed the discrete analogue of Voronin's theorem:

{\it Let $h>0$. If $(K,f,\varepsilon)^*$ is admissible, then
\begin{align}\label{discreteVoroninuniversailty}
	\liminf_{N\to\infty}\frac{1}{N}\#\left\{n\leq N:\max_{s\in K}|\zeta(s+ihn)-f(s)|<\varepsilon\right\}>0,
\end{align}
where $\# A$ denotes the cardinality of the set $A\subset \mathbb{N}$.}

In other words, the set $\lbrace\zeta(\cdot+ihn):n\geq1\rbrace$ intersects frequently often any open subset of $H^*(\mathcal{D})$.
In this case we talk about {\it $h$-discrete universality}  while for the Hurwitz zeta-function one has {\it $h$-discrete strong universality} (see for example \cite{LaurincikasMacaitiene2009}).

\subsection{Main results}
The proofs of continuous and $h$-discrete (strong) universality coincide up to some point and then it seems that the two methods, though similar, are not directly connected.
As a matter of fact, in most cases (including the one of the Hurwitz zeta-function) the $h$-discrete universality is obtained for all $h>0$ except for a set of zero Lebesgue measure.
As a consequence there are many results  which treat these phenomena separately. 
Particularly, if the universality property for a zeta- or an $L$-function is established, then shortly after its discrete analogue appears in the literature as well. 
So it is natural to ask if there is a connection between these concepts and whether one of them implies the other.
We will see that if we translate this question in the language of linear dynamics, then there is an affirmative answer which we can employ to prove the following theorems.
\begin{Theorem}\label{ContDiscHurw}
	Let $h>0$. 
	The continuous strong universality of the Hurwitz zeta-function implies its $h$-discrete strong universality.
	Conversely, the $h$-discrete strong universality implies \eqref{Hurwitzuniversal} with $\limsup$ in place of $\liminf$.
\end{Theorem}
\begin{Theorem}\label{ContDiscRiemacon}
	Let $h>0$.
	Assuming the Riemann hypothesis, the continuous universality of the Riemann zeta-function is equivalent to its $h$-discrete universality.
	\end{Theorem}
It will be seen that in the case of strong universality the theorem is in fact a consequence of the {Conejero-M\"uller-Peris theorem} and a classic result on strong recurrence due to Gottschalk and Hedlund.
The theorem for the Riemann zeta-function will follow by the Conejero-M\"uller-Peris theorem and an equivalent formulation of the Riemann hypothesis with the strong recurrence property of $\zeta(s)$ due to Bagchi.
In the end we present some heuristics on why the Conejero-M\"uller-Peris theorem is not applicable if we assume the existence of hypothetical non-trivial zeros of $\zeta(s)$ off the vertical line $1/2+i\mathbb{R}$.
Nevertheless, we will employ an older of result due to Oxtoby and Ulam to obtain a  much weaker result for $\zeta(s)$ unconditionally.
\begin{Theorem}\label{ContDiscRiem}
	The continuous universality of the Riemann zeta-function implies the existence of a dense $G_\delta$-set $J\subseteq\mathbb{R}_+$ such that if $t_0\in J$, then for any admissible tuple $(K,f,\varepsilon)^*$, there is a sequence $(n_k)_{k\geq1}\subseteq\mathbb{N}$ such that
	\[
	\max_{s\in K}|\zeta(s+it_0n_k)-f(s)|<\varepsilon,\quad k\geq1.
	\]
\end{Theorem}

The proofs of the theorems are based on properties of dynamical systems.
Hence, we prove them for the simplest cases of $\zeta(s)$ and $\zeta(s;\alpha)$ with the purpose of exhibiting the method.
The main point of the present work is to show that any form of continuous universality for zeta-functions implies in one way or the other its discrete analogue.
\subsection{Structure and notations}
In Section \ref{dynam} we present the necessary material from the theory of linear dynamics that will be repeatedly employed in the sequel.
In Section \ref{Proofs} we show how the questions we want to address on zeta-functions can be translated in the language of linear dynamics and provide the proofs of the main results.
In the last two sections we state shortly further generalizations of the method introduced in earlier sections.
\section{Linear dynamics and strong recurrence}\label{dynam}
The interested reader can turn to \cite{Bayartbook,Erdmanbook} for a contemporary treatment of the theory of linear dynamics.
 
In the sequel  $X$ will denote a topological vector space over $\mathbb{C}$.
A continuous linear map $T:X\to X$ is called an {\it operator} and a vector $x\in X$ is called a {\it hypercyclic vector of $T$} if its orbit
\[
\mathrm{orb}(x,T):=\left\{T^nx:n\geq0\right\},
\]
is dense in $X$; here and in the sequel we use the usual convention from operator theory $Tx:=T(x)$ and 
\[
T^nx:=\underbrace{T\circ\dots\circ T}_{n\text{ times}}x.
\]
The set of hypercyclic vectors  is denoted by $\mathrm{HC}(T)$.
Hence, a vector $x$ is hypercyclic if for every open set $U\subseteq X$, there is $n\in\mathbb{N}$ such that $T^nx\in U$. 
If additionally we can show that
\[
\liminf_{N\to\infty}\frac{1}{N}\#\left\{n\leq N:T^nx\in U\right\}>0
\]
for every open set $U\subseteq X$, then $x$ is called a {\it frequently hypercyclic vector of $T$} and the set of frequently hypercyclic vectors will be denoted by $\mathrm{FHC}(T)\subseteq\mathrm{HC}(T)$.

In the continuous case, a one-parameter family $\mathcal{T} = (T_t)_{\tau\geq0}$ of operators
defined on $X$ is called a {\it strongly continuous semigroup} (or $C_0$-semigroup) if $T_0 = I$,
$T_t T_s = T_{t+s}$ for all $t, s \geq0$, and $\lim_{t\to s} T_t x = T_s x$ for all $s \geq 0$, $x \in X$.
In this setting, a vector $x\in X$ is called a {\it hypercyclic vector of $\mathcal{T}$} if $\mathrm{orb}(\mathcal{T} , x) := \left\{T_tx: t\geq0\right\}$ is dense in $X$ and a {\it frequently hypercyclic vector of} $\mathcal{T}$ if
\[
\liminf_{T\to\infty}\frac{1}{T}\mathrm{meas}\left\{t\in[0,T]:T_tx\in U\right\}>0
\]
for every open set $U\subseteq X$.
Completely analogously as in the discrete case, the set of hypercyclic, respect. frequently hypercyclic, vectors of the family $\mathcal{T}$ is denoted by $\mathrm{HC}(\mathcal{T})$, respect. $\mathrm{FHC}(\mathcal{T})\subseteq\mathrm{HC}(\mathcal{T})$. 

From the above definitions, it follows that $\mathrm{HC}(T_t)\subseteq\mathrm{HC}(\mathcal{T})$ for every $t>0$ because $T_t^n=T_{nt}$.
Whether the converse implication $\mathrm{HC}(\mathcal{T})\subseteq \mathrm{HC}(T_t)$ is true for some $t>0$, has been first answered partially by Oxtoby and Ulam \cite{OxtobyUlam}.
We also refer to \cite[Theorem 7.22]{Erdmanbook} for a proof.

\begin{Theorem*}[Oxtoby-Ulam]
	If $\mathcal{T}=(T_t)_{t\geq0}$ is a $C_0$-semigroup on a separable space $X$ and
	$x \in \mathrm{HC}(\mathcal{T})$, then there is a dense $G_\delta$-set $J \subseteq (0,+\infty)$ such $x\in \mathrm{HC}(T_t)$ for any $t \in J$.
\end{Theorem*}
Whether one has $\mathrm{HC}(\mathcal{T})\subseteq \mathrm{HC}(T_t)$ for every $t>0$, it has been answered recently by Conejero, M\"uller and Peris \cite{ConejeroMuellerPeris}.
\begin{Theorem*}[Conejero-M\"uller-Peris]
If $\mathcal{T}=(T_t)_{t\geq0}$ is a $C_0$-semigroup of operators that is locally equicontinuous,
then $\mathrm{HC}(\mathcal{T})\subseteq\mathrm{HC}(T_h)$ and $\mathrm{FHC}(\mathcal{T})\subseteq\mathrm{FHC}(T_h)$ for any $h>0$.
\end{Theorem*}
We close the section with the notion of strong recurrence.
A vector $x\in X$ will be called {\it strongly recurrent} for the operator $T$, respectively for the $C_0$-semigroup $\mathcal{T}$, if for every open set $U\subseteq X$ with $x\in U$ 
\[
\limsup_{N\to\infty}\frac{\#\left\{n\leq N:T^nx\in U\right\}}{N}>0,
\]
respectively
\[\limsup_{T\to\infty}\frac{\mathrm{meas}\left\{t\in[0,T]:T_tx\in U\right\}}{T}>0.
\]
There is a connection between continuous and discrete version of strong recurrence which was established  by Gottschalk and Hedlund \cite[Theorem 2]{GottschalkHedlund}.

\begin{Theorem*}[Gottschalk-Hedlund]
	Let $\mathcal{T}=(T_t)_{t\geq0}$ be a $C_0$-semigroup of operators and assume that $h>0$. 
	Then a vector $x\in X$ is strongly recurrent for $\mathcal{T}$ if and only if is strongly recurrent for $T_h$.
\end{Theorem*}
The notion of strong recurrence in the context of zeta- and $L$-functions has been realized by Bagchi in his thesis \cite{Bagchithesis} and then in subsequent work in \cite{Bagchi1982, Bagchi1987}.
\begin{Theorem*}[Bagchi]
The following statements are equivalent:
\begin{enumerate}
	\item The Riemann hypothesis is true.
	\item For any admissible $(K,\zeta,\varepsilon)$
	\[
	\liminf_{T\to\infty}\frac{1}{T}\mathrm{meas}\left\{\tau\in[0,T]:\max_{s\in K}|\zeta(s+i\tau)-\zeta(s)|<\varepsilon\right\}>0.
	\]
	\item Statement $(2)$ holds with $\limsup$ in place of $\liminf$.
\end{enumerate}
\end{Theorem*}
It should be noted that  Bagchi proves only $(1)\Leftrightarrow(3)$ but the statement cited mostly in the relevant literature is $(1)\Leftrightarrow(2)$.
We give a sketch of the proof. 
The implication $(1)\Rightarrow(2)$ follows from Voronin's theorem while $(2)\Rightarrow(3)$ is obvious.
Lastly, assuming that $(3)$ holds but $(1)$ does not, it would imply by Rouch\'e's theorem the existence of at least $cT$ many zeros in the rectangle $1/2<\sigma<1$, $0<t<T$, for some $c>0$ and infinitely many $T$.
But these would contradict classic zero-density estimates of $\zeta(s)$ in the half-plane $\sigma>1/2$ (see for example \cite[Chapter IX]{Titchbook}).

\section{Proofs of the main results}\label{Proofs}
Before giving the proofs we construct the space $X$ and the one-parameter family of operators $\mathcal{T}$ which are connected to the universality of zeta-functions.

Let $(K_n)_{n\geq1}$ be an exhaustion of $\mathcal{D}=\left\{s\in\mathbb{C}:1/2<\sigma<1\right\}$ by compact sets, that is $K_n\subseteq K_{n+1}\subseteq \mathcal{D}$, $n\geq1$, and for every compact set $K\subseteq U$, there is $N\geq1$ such that $K\subseteq K_N$.
Since $\mathcal{D}$ is simply connected we can construct $K_n$ to have connected complement.
We  equip $H(\mathcal{D})=\left\{f:\mathcal{D}\to\mathbb{C}\mid f\text{ is analytic}\right\}$ with the sequence of norms
\[
p_n(f):=\max_{s\in K_n}|f(s)|,\quad f\in H(\mathcal{D}),\quad n\geq1,
\]
and it becomes a (separable) Fr\'echet space, i.e. a complete metric space with the metric
\[
d(f,g):=\sum_{n\geq1}\frac{1}{2^n}\min\left(1,p_n(f-g)\right),\quad f,g\in H(\mathcal{D}).
\]
Lastly, we define the sequence of {\it translation operators}  $\mathcal{T}:=(T_\tau)_{\tau\geq0}$ by
\[
T_\tau f:=f(\cdot+i\tau), \quad f\in H(\mathcal{D}),\quad\tau\geq0,
\]
and it can be quickly verified that it is a well-defined $C_0$-semigroup on $H(\mathcal{D})$.
The family is also locally equicontinuous because $H(\mathcal{D})$ is a Fr\'echet space.

\begin{proof}[Proof of Theorem \ref{ContDiscHurw}]
	Let $h>0$.
	If $U$ is an open subset of $H(\mathcal{D})$, then there is $n\geq1$ and an analytic function $h:\mathcal{D}\to\mathbb{C}$ such that
	\[
	V:=\left\{g\in H(\mathcal{D}):\max_{s\in K_n}|g(s)-h(s)|<\frac{1}{n}\right\}\subseteq U.
	\]
If the Hurwitz zeta-function  is continuous strongly universal for some fixed parameter $\alpha\in(0,1]$, then
	\[
	\liminf_{T\to\infty}\frac{1}{T}\mathrm{meas}\left\{\tau\in[0,T]:\max_{s\in K_n}|\zeta(s+i\tau;\alpha)-h(s)|<\frac{1}{n}\right\}>0,
	\] 
	which implies in the language of linear dynamics that
	\[
	\liminf_{T\to\infty}\frac{1}{T}\mathrm{meas}\left\{\tau\in[0,T]:T_\tau\zeta(\cdot;\alpha)\in U\right\}>0.
	\]
	Since $U$ is arbitrary and $\zeta(\cdot;\alpha)$ is an element of $H(\mathcal{D})$, we get that $\zeta(\cdot;\alpha)\in\mathrm{FHC}(\mathcal{T})$.
	In view of the Conejero-M\"uller-Peris theorem it follows that also $\zeta(\cdot;\alpha)\in\mathrm{FHC}(T_h)$.
	
	Let now $(K,f,\varepsilon)$ be admissible.
	By Mergelyan's theorem there is a polynomial $P(s)$ with
	\[
	\max_{s\in K}|P(s)-f(s)|<\frac{\varepsilon}{2}.
	\]
	On the other hand 
	\[
	W:=\left\{g\in H(\mathcal{D}):\max_{s\in K}|g(s)-P(s)|<\frac{\varepsilon}{2}\right\}
	\]
	is an open subset of $H(\mathcal{D})$.
	Since  $\zeta(\cdot;\alpha)\in\mathrm{FHC}(T_h)$, we have that
	\[
	\liminf_{N\to\infty}\frac{1}{N}\#\left\{n\leq N:T^n_h\zeta(\cdot;\alpha)\in W\right\}.
	\]
	Therefore, by the triangle inequality and  the equation $T^n_h=T_{hn}$ we conclude that
	\[
	\liminf_{N\to\infty}\frac{1}{N}\#\left\{n\leq N:\max_{s\in K}|\zeta(s+ihn;\alpha)-f(s)|<\varepsilon\right\}>0.
	\]
	Hence, $\zeta(s;\alpha)$ is $h$-discrete strongly universal as well.
	
	Now we assume that the converse statement is true, i.e. the Hurwitz zeta-function is $h$-discrete strongly universal.
	If $U\subseteq H(\mathcal{D})$ is open with $\zeta(\cdot,\alpha)\in U$, then by assumption
	\[
	\liminf_{N\to\infty}\frac{1}{N}\left\{n\leq N:T_h^n\zeta(\cdot;\alpha)\in U\right\}.
	\]
	Hence, $\zeta(\cdot;\alpha)$ is strongly recurrent for the operator $T_h$ and, consequently, for the family $\mathcal{T}$ as can be seen by the Gottshalk-Hedlund theorem.
	
If $(K,f,\varepsilon)$ is admissible, then the $h$-discrete strong universality implies that there is integer $N$ such that
	\begin{align}\label{discretaprr}
	\max_{s\in K}|\zeta(s+ihN;\alpha)-f(s)|<\frac{\varepsilon}{2}.
	\end{align}
	Since $\zeta(\cdot;\alpha)$ is strongly recurrent for $\mathcal{T}$, we know that
	\[
	\limsup_{T\to\infty}\frac{1}{T}\mathrm{meas}\left\{\tau\in[0,T]:\max_{s\in ihN+K}|\zeta(s+i\tau;\alpha)-\zeta(s;\alpha)|<\frac{\varepsilon}{2}\right\}>0.
	\]
	From the triangle inequality and \eqref{discretaprr}, we obtain now the second statement of the theorem.
\end{proof}
\begin{proof}[Proof of Theorem \ref{ContDiscRiemacon}]
	Let $h>0$.
	It should be noted that  if $(K,f,\varepsilon)$ is admissible ($f$ may also have zeros), then
	\begin{align}\label{logRie}
	\liminf_{T\to\infty}\frac{1}{T}\mathrm{meas}\left\{\tau\in[0,T]:\max_{s\in K}|\log\zeta(s+i\tau)-f(s)|<\varepsilon\right\}.
	\end{align}
	For, we know that hypothetical zeros of $\zeta(s)$ off the vertical line $1/2+i\mathbb{R}$ are distributed sparser as we move higher in the strip $\mathcal{D}$.
	Therefore, for any compact set $K$, it is possible to show that the measure of those $\tau\in[0,T]$ such that $i\tau+K$ does not contain any zeros of $\zeta(s)$, is asymptotically equal to $T$.
	On the other hand, Voronin's theorem implies that the set of $\tau$ such that $\zeta(s+i\tau)$ is close to $e^{f(s)}$ has positive lower density.
	Relation \eqref{logRie} follows by combining these two facts.
	
	It is only now that we have to assume the Riemann hypothesis in order to ensure that $\log\zeta\in H(\mathcal{D})$.
	We then have the same setting as for the Hurwitz zeta-function in the previous theorem.
	The continuous universality of $\zeta(s)$ implies the continuous strong universality of $\log\zeta(s)$ which in turn (by the Conejero-M\"uller-Peris theorem) implies the $h$-discrete strong universality of $\log\zeta(s)$ and by exponentiation the $h$-discrete universality of $\zeta(s)$.
	
	For the converse statement however, we employ additionally Bagchi's theorem after applying the Gottshalk-Hedlund theorem. 
	Of course it is again essential to ensure that $\log\zeta\in H(\mathcal{D})$ as follows from the Riemann hypothesis.
	By repeating now the same steps as in the previous proof but moving from $\zeta(s)$ to its logarithm and  preserving the $\liminf$ notation, we obtain that  the $h$-discrete universality of $\zeta(s)$ implies its continuous universality. 
\end{proof}

\begin{Remark*}
One of the main arguments in the proof of the Conejero- M\"uller-Peris theorem is that the (frequently) hypercyclic vector can approximate itself in the space $X$.
This may not be possible in our case unless we assume the Riemann hypothesis in which case Voronin's theorem is applicable.
Another approach relies on working in a slit half-plane $\Omega$ which is  defined as the whole $\mathcal{D}$ without the segments $(1/2+i\gamma,\beta+i\gamma]$ for any zero $\rho=\beta+i\gamma$ of $\zeta(s)$.
Since $\zeta(1+it)\neq0$, $t\in\mathbb{R}$, $\Omega$ is simply connected and we can show as in the case of $H(\mathcal{D})$, that $H(\Omega)$ is a Fr\'echet space with the induced topology.
The disadvantage of this approach is that the family $\mathcal{T}$ of vertical shifts in not anymore well-defined in $H(\Omega)$.
Nevertheless, if we leave aside the notion of operators then Voronin's theorem implies the strong recurrence of $\zeta(s)$ in the space $H(\Omega)$ (in the form given by Bagchi).
However, later in the proof of the Conejero-M\"uller-Peris theorem, a certain continuity argument is needed that involves the elements $x+\lambda T_\tau x$, $x\in HC(\mathcal{T})$, $\lambda\in\mathbb{C}$ and $\tau$ from some suitable interval.
This argument does not seem that can be adjusted to our case if we drop some defining properties of a $C_0$-semigroup.
Instead, to prove the weaker statement of Theorem \ref{ContDiscRiem}, we modify the proof of \cite[Theorem 7.22]{Erdmanbook}, where no self-approximation is needed.
\end{Remark*}
\begin{proof}[Proof of Theorem \ref{ContDiscRiem}]
	Let $U_m$, $m\geq1$, be an enumeration of the sets
	\[
	\left\{f\in H^*(\mathcal{D}):\max_{s\in K_N}|f(s)-e^{P(s)}|<\frac{1}{N}\right\},\quad N\geq1, \,P\in\mathbb{Q}[X],
	\]
which form a countable base of $H^*(\mathcal{D})$ and for each $m\geq1$ define
	\begin{align*}
		J_m=\left\{\tau\in(0,+\infty):\exists n\in\mathbb{N}\text{ with }\max_{s\in K_N}|\zeta(s+in\tau)-e^{P(s)}|<\frac{1}{N}\right\},
	\end{align*}
	which is an open subset of $(0,+\infty)$.

	By Voronin's theorem it follows that $J_m$ is dense.
For, if $0<a<b<+\infty$, then there is $n_0\in\mathbb{N}$ such that $n_0b>(n_0+1)a$ and, thus, $(n_0a,+\infty)\subseteq\cup_{n\geq n_0}(na,nb)$.
The universality theorem implies that $T_s\zeta\in U_m$ for (infinitely many) $s\geq n_0$.
Hence, $s\in(na,nb)$ for some $n\geq n_0$ and if we set $t_0=s/n$, then $t_0\in J_m\cap(a,b)$.

Therefore, the Baire category theorem yields that the set
	\[
	J:=\bigcap_{m\geq1}J_m
	\]
is a dense $G_\delta$-set in $(0,+\infty)$ whose elements satisfy the desired property of $\zeta(s)$.
\end{proof}
\section{Multidimensional case}
There is no real advantage on considering only $\zeta(s)$ or $\zeta(s;\alpha)$ since any other function defined in a similar space as $H(\mathcal{D})$ could take their place.
In fact, we introduce the following generalization, which has also been realized by Bagchi \cite{Bagchi1987}:

Let $N\in\mathbb{N}$, $\mathbf{h}:=(h_1,\dots,h_N)$ be a vector of positive real numbers and $\sigma_{1n}<\sigma_{2n}$, $n\leq N$.
If $\mathcal{D}_n:=\left\{s\in\mathbb{C}:\sigma_{1n}<\sigma<\sigma_{2n}\right\}$, $n\leq N$,  and $\mathcal{C}:=\prod_{n\leq N}\mathcal{D}_n$ is their Cartesian product, then $H(\mathcal{C})$ is a separable Fr\'echet space (with the product topology) and the family of operators $\mathcal{T}_\mathbf{h}=(T_\tau)_{\tau\geq0}$ defined by
\[
T_\tau(f_1,\dots,f_N)=(f_1(\cdot+ih_1\tau),\dots,f_N(\cdot+ih_N\tau)),\quad(f_1,\dots,f_N)\in H(\mathcal{C}),\quad\tau\geq0,
\]
is a well-defined $C_0$-semigroup.
Lastly, $H^*(\mathcal{C})\subseteq H(\mathcal{C})$ will consist of the vectors whose entries are zero-free functions.

A vector of {\bf zeta-functions} $(\zeta_1,\dots,\zeta_N)\in H(\mathcal{C})$ will be called:
\begin{enumerate}
	\item {\it continuous jointly  universal} if for any admissible $(K_n,f_n,\varepsilon)^*$, $n\leq N$,
	\begin{align}\label{jointunivers}
	\liminf_{T\to\infty}\frac{1}{T}\mathrm{meas}\left\{\tau\in[0,T]:\max_{n\leq N}\max_{s\in K_i}|\zeta_n(s+ih_n\tau)-f_n(s)|<\varepsilon\right\}>0;
	\end{align}
	\item {\it continuous joint strongly universal} if for any admissible $(K_n,f_n,\varepsilon)$, $n\leq N$, relation \eqref{jointunivers} holds;
	\item {\it $h$-discrete jointly universal/joint strongly universal} if in \eqref{jointunivers} we substitute the continuous variable $\tau\in[0,T]$ with a discrete variable $hm$, $m\in[0,T]\cap\mathbb{N}$, and the Lebesgue measure notation $\mathrm{meas}$ with the notation of the cardinality of a set $\#$.
\end{enumerate}
	With the above notations, we can adjust the proofs from the previous sections to obtain similar results for the multidimensional case.
\begin{Theorem}\label{multidimensional}
	If $(\zeta_1,\dots,\zeta_N)\in H(\mathcal{C})$ is continuous joint strongly universal, then it is $h$-discrete joint strongly universal for every $h>0$.
	If $(\zeta_1,\dots,\zeta_N)\in H^*(\mathcal{C})$ is continuous jointly universal, then it is $h$-discrete joint universal for every $h>0$.
	If $(\zeta_1,\dots,\zeta_N)\in H(\mathcal{C})$ is continuous jointly universal, then there is a dense $G_\delta$-set $J\subseteq(0,+\infty)$ such that if $t_0\in J$, then for any admissible tuple $(K_n,f_n,\varepsilon)^*$, $n\leq N$, there is a sequence $(n_k)_{k\geq1}\subseteq\mathbb{N}$ such that
	\[
	\max_{n\leq N}\max_{s\in K_n}|\zeta_n(s+it_0n_k)-f_n(s)|<\varepsilon,\quad k\geq1.
	\]
\end{Theorem}
\section{Zeta-functions}
The reason we highlighted the term ``zeta-functions'' is because in analytic number theory there is no rigorous definition of what a zeta- or an $L$-function should look like.
They are usually considered  as Dirichlet series
\[
L(s)=\sum_{n\geq1}\frac{a_n}{e^{\lambda_ns}},\quad\sigma>\sigma_0,
\]
where $(a_n)_{n\geq1}\subseteq\mathbb{C}$ and $(\lambda_n)_{n\geq1}\subseteq\mathbb{R}$ are sequences of number-theoretic interest (for example $\lambda_n=\log n$ and $a_n=1$) and $\sigma_0<\infty$ is the abscissa of absolute convergence.
Naturally,  $L(s)$ can not be universal in the half-plane $\sigma>\sigma_0$.
If it can, however, be analytically continued to a vertical strip $\sigma_1<\sigma<\sigma_0$ with the exception of finitely many poles, then it will most likely be universal as well.
Here we can make two distinctions.
If $L(s)$ has also an Euler product representation
\[
L(s)=\prod_{n\geq1}\parentheses*{1-\frac{b_n}{e^{\mu_n s}}}^{-1},\quad\sigma>\sigma_0,
\]
for some $(b_n)_{n\geq1}\subseteq\mathbb{C}$ and $(\mu_n)_{n\geq1}\subseteq\mathbb{R}$, then similar zero-density estimates as in the case of $\zeta(s)$ could possibly be attained, allowing us to say that $L(s)$ can not be strongly universal.
If, on the other hand, such representation does not exist, then $L(s)$ will be strongly universal.

A good candidate of a zeta-function having an Euler product representation are elements from the so-called {\it Selberg class} introduced by Selberg \cite{Selberg1992old}. 
For instance, the Riemann zeta-function, Dirichlet $L$-functions, Dedekind zeta-functions and Hecke $L$-functions belong to this class.
For a detailed survey we refer to \cite{Steudingbook}, while sufficient conditions for an $L(s)$ from this class to be universal are given in \cite{Nagoshisteuding}.
If additionally we assume some sort of independence between elements $L_1,\dots,L_N$ of the Selberg class, then we can also have joint universality.
For example if $\chi_1,\dots,\chi_N$ are Dirichlet characters that are pairwise nonequivalent, then the associated Dirichlet $L$-functions $(L(s,\chi_1),\dots, L(s,\chi_N))$ are jointly universal.
This was proved independently by Bagchi \cite{Bagchithesis,Bagchi1982},  Gonek \cite{Gonekthesis} and Voronin \cite{Voroninthesis}.
A more general framework in the context of the Selberg class is given in \cite{Leenakamura2017}.

Zeta-functions without an Euler product are usually occurring when $\lambda_n=\log\kappa_n$ with $\kappa_n\in\mathbb{R}_+\setminus\mathbb{N}$ or when they can be expressed as a linear combination of two or more zeta-functions which have an Euler product.
A classic example of the first case is the Lerch zeta-function \cite{Laurincikasbook} which is a generalization of the Hurwitz zeta-function, while an example of the second case are Dirichlet series with periodic coefficients \cite{Steuding2002}.
In both cases we have strong universality while in \cite{Bayart2023,Laurincikas2003} a more general framework is given.
If a tuple of such zeta-functions has some sort of independence between them, then they will also be joint strongly universal \cite{Javtokas2008,Lee2017}.

We only presented a selection of results.
Their discrete analogues can also be found in the literature.
On the other hand, Theorem \ref{multidimensional} implies in a strong or a weaker sense that studying the continuous universality may suffice.
\bibliographystyle{plainnat}
\bibliography{bibliography}

\begin{thebibliography}{27}
\providecommand{\natexlab}[1]{#1}
\providecommand{\url}[1]{\texttt{#1}}
\expandafter\ifx\csname urlstyle\endcsname\relax
  \providecommand{\doi}[1]{doi: #1}\else
  \providecommand{\doi}{doi: \begingroup \urlstyle{rm}\Url}\fi

\bibitem[Bagchi(1981)]{Bagchithesis}
B.~Bagchi.
\newblock \emph{Statistical behaviour and universality properties of the
  Riemann zeta function and other allied Dirichlet series}.
\newblock PhD thesis, Indian Statistical Institute-Kolkata, 1981.

\bibitem[Bagchi(1982)]{Bagchi1982}
B.~Bagchi.
\newblock A joint universality theorem for {Dirichlet} {L}-functions.
\newblock \emph{Math. Z.}, {\bf 181}:\penalty0 \,319--334, 1982.
\newblock \doi{10.1007/BF01161980}.

\bibitem[Bagchi(1987)]{Bagchi1987}
B.~Bagchi.
\newblock Recurrence in topological dynamics and the {Riemann} hypothesis.
\newblock \emph{Acta Math. Hung.}, {\bf 50}:\penalty0 \,227--240, 1987.
\newblock \doi{10.1007/BF01903937}.

\bibitem[Bayart(2023)]{Bayart2023}
F.~Bayart.
\newblock Universality of general {Dirichlet} series with respect to
  translations and rearrangements.
\newblock \emph{Ark. Mat.}, {\bf 61}\penalty0 (1):\penalty0 \,19--39, 2023.
\newblock \doi{10.4310/ARKIV.2023.v61.n1.a2}.

\bibitem[Bayart and Matheron(2009)]{Bayartbook}
F.~Bayart and {\'E}.~Matheron.
\newblock \emph{Dynamics of linear operators}, volume 179 of \emph{Camb. Tracts
  Math.}
\newblock Cambridge: Cambridge University Press, 2009.

\bibitem[Conejero et~al.(2007)Conejero, M{\"u}ller, and
  Peris]{ConejeroMuellerPeris}
J.~A. Conejero, V.~M{\"u}ller, and A.~Peris.
\newblock Hypercyclic behaviour of operators in a hypercyclic
  {{\(C_{0}\)}}-semigroup.
\newblock \emph{J. Funct. Anal.}, {\bf 244}\penalty0 (1):\penalty0 \,342--348,
  2007.
\newblock \doi{10.1016/j.jfa.2006.12.008}.

\bibitem[Gonek(1979)]{Gonekthesis}
S.~M. Gonek.
\newblock \emph{Analytic properties of zeta and $L$-functions}.
\newblock PhD thesis, 1979.

\bibitem[Gottschalk and Hedlund(1946)]{GottschalkHedlund}
W.~H. Gottschalk and Gustav~A. Hedlund.
\newblock Recursive properties of transformation groups.
\newblock \emph{Bull. Am. Math. Soc.}, 52:\penalty0 637--641, 1946.
\newblock \doi{10.1090/S0002-9904-1946-08612-7}.

\bibitem[Grosse-Erdmann and P.~Manguillot(2011)]{Erdmanbook}
K.-G. Grosse-Erdmann and A.~P.~Manguillot.
\newblock \emph{Linear chaos}.
\newblock Universitext. Berlin: Springer, 2011.
\newblock \doi{10.1007/978-1-4471-2170-1}.

\bibitem[Javtokas and Laurin{\v{c}}ikas(2008)]{Javtokas2008}
A.~Javtokas and A.~Laurin{\v{c}}ikas.
\newblock A joint universality theorem for periodic {Hurwitz} zeta-functions.
\newblock \emph{Bull. Aust. Math. Soc.}, {\bf 78}\penalty0 (1):\penalty0
  \,13--33, 2008.
\newblock \doi{10.1017/S0004972708000300}.

\bibitem[Laurin{\v{c}}ikas(1995)]{Laurincikasbook}
A.~Laurin{\v{c}}ikas.
\newblock \emph{Limit theorems for the {Riemann} zeta-function}, volume 352 of
  \emph{Math. Appl., Dordr.}
\newblock Dordrecht: Kluwer Academic Publishers, 1995.

\bibitem[Laurin{\v{c}}ikas and Macaitien{\.e}(2009)]{LaurincikasMacaitiene2009}
A.~Laurin{\v{c}}ikas and R.~Macaitien{\.e}.
\newblock The discrete universality of the periodic {Hurwitz} zeta function.
\newblock \emph{Integral Transforms Spec. Funct.}, {\bf 20}\penalty0
  (9-10):\penalty0 \,673--686, 2009.
\newblock \doi{10.1080/10652460902742788}.

\bibitem[Laurin{\v{c}}ikas et~al.(2003)Laurin{\v{c}}ikas, Schwarz, and
  Steuding]{Laurincikas2003}
A.~Laurin{\v{c}}ikas, W.~Schwarz, and J.~Steuding.
\newblock The universality of general {Dirichlet} series.
\newblock \emph{Analysis, M{\"u}nchen}, {\bf 23}\penalty0 (1):\penalty0
  \,13--26, 2003.

\bibitem[Lee et~al.(2017{\natexlab{a}})Lee, Nakamura, and
  Pa{\'n}kowski]{Lee2017}
Y.~Lee, T.~Nakamura, and {\L}.~Pa{\'n}kowski.
\newblock Joint universality for {Lerch} zeta-functions.
\newblock \emph{J. Math. Soc. Japan}, {\bf 69}\penalty0 (1):\penalty0
  \,153--161, 2017{\natexlab{a}}.
\newblock \doi{10.2969/jmsj/06910153}.

\bibitem[Lee et~al.(2017{\natexlab{b}})Lee, Nakamura, and
  Pa{\'n}kowski]{Leenakamura2017}
Y.~Lee, T.~Nakamura, and {\L}.~Pa{\'n}kowski.
\newblock Selberg's orthonormality conjecture and joint universality of
  {{\(L\)}}-functions.
\newblock \emph{Math. Z.}, 286\penalty0 (1-2):\penalty0 1--18,
  2017{\natexlab{b}}.
\newblock \doi{10.1007/s00209-016-1754-2}.

\bibitem[Matsumoto(2015)]{Matsumotosurvey}
K.~Matsumoto.
\newblock A survey on the theory of universality for zeta and
  {{\(L\)}}-functions.
\newblock In \emph{Number theory. Plowing and starring through high wave forms.
  Proceedings of the 7th China-Japan Seminar, Fukuoka, Japan, October 28 --
  November 1, 2013}, pages \,95--144. Hackensack, NJ: World Scientific, 2015.
\newblock \doi{10.1142/9789814644938_0004}.

\bibitem[Nagoshi and Steuding(2010)]{Nagoshisteuding}
H.~Nagoshi and J.~Steuding.
\newblock Universality for {{\(L\)}}-functions in the {Selberg} class.
\newblock \emph{Lith. Math. J.}, {\bf 50}\penalty0 (3):\penalty0 \,293--311,
  2010.
\newblock \doi{10.1007/s10986-010-9087-z}.

\bibitem[Oxtoby and Ulam(1941)]{OxtobyUlam}
J.~C. Oxtoby and S.~M. Ulam.
\newblock Measure-preserving homeomorphisms and metrical transitivity.
\newblock \emph{Ann. Math. (2)}, {\bf 42}:\penalty0 \,874--920, 1941.
\newblock \doi{10.2307/1968772}.

\bibitem[Queffelec and Queffelec(2020)]{Queffelecbook}
H.~Queffelec and M.~Queffelec.
\newblock \emph{Diophantine approximation and {Dirichlet} series}, volume~80 of
  \emph{Texts Read. Math.}
\newblock New Delhi: Hindustan Book Agency; Singapore: Springer, 2nd extended
  edition edition, 2020.
\newblock \doi{10.1007/978-981-15-9351-2}.

\bibitem[Reich(1980)]{Reich1980}
A.~Reich.
\newblock Werteverteilung von {Zetafunktionen}.
\newblock \emph{Arch. Math.}, {\bf 34}:\penalty0 \,440--451, 1980.
\newblock \doi{10.1007/BF01224983}.

\bibitem[Riemann(2017)]{Riemann}
B.~Riemann.
\newblock \emph{The collected works of {Bernhard} {Riemann}. {The} complete
  {German} texts. {With} an {English} introduction by {Hans} {Lewy}}.
\newblock Mineola, NY: Dover Publications, reprint of the 1953 {Dover}
  republication of the 2nd edition of {Bernhard} {Riemann}'s `{Gesammelte}
  {Mathematische} {Werke}' published in 1892 and the {Supplement}
  ({Nachtr{\"a}ge}) published 1902 edition, 2017.
\newblock ISBN 978-0-486-81243-4.

\bibitem[Selberg(1992)]{Selberg1992old}
A.~Selberg.
\newblock Old and new conjectures and results about a class of dirichlet
  series.
\newblock In \emph{Proceedings of the Amalfi Conference on Analytic Number
  Theory (Maiori, 1989)}, volume~2, pages 47--63, 1992.

\bibitem[Steuding(2002)]{Steuding2002}
J.~Steuding.
\newblock On {Dirichlet} series with periodic coefficients.
\newblock \emph{Ramanujan J.}, {\bf 6}\penalty0 (3):\penalty0 \,295--306, 2002.
\newblock \doi{10.1023/A:1019797315282}.

\bibitem[Steuding(2007)]{Steudingbook}
J.~Steuding.
\newblock \emph{Value distribution of {{\(L\)}}-functions}, volume 1877 of
  \emph{Lect. Notes Math.}
\newblock Berlin: Springer, 2007.
\newblock \doi{10.1007/978-3-540-44822-8}.

\bibitem[Titchmarsh(1986)]{Titchbook}
E.~C. Titchmarsh.
\newblock \emph{The theory of the {Riemann} zeta-function. 2nd ed., rev. by
  {D}. {R}. {Heath}-{Brown}}.
\newblock 1986.

\bibitem[Voronin(1975)]{Voronin1975}
S.~M. Voronin.
\newblock Theorem on the `universality' of the {Riemann} zeta-function.
\newblock \emph{Izv. Akad. Nauk SSSR, Ser. Mat.}, {\bf39}:\penalty0 \,475--486,
  1975.

\bibitem[Voronin(1978)]{Voroninthesis}
S.~M. Voronin.
\newblock \emph{Analytic properties of Dirichlet generating functions of
  arithmetic objects}.
\newblock PhD thesis, 1978.

\end{thebibliography}
\end{document}